\documentclass[conference]{IEEEtran}
\IEEEoverridecommandlockouts

\usepackage{cite}
\usepackage{amsmath}
\usepackage{amsthm}
\usepackage[left=48pt, right=48pt, top=57pt, bottom=57pt]{geometry}
\usepackage{amssymb,amsfonts}
\usepackage{mathtools}
\usepackage{algorithmic}
\usepackage{graphicx}
\usepackage{textcomp}
\usepackage{caption}
\usepackage{subcaption}
\newtheorem{theorem}{Theorem}
\newtheorem{lemma}{Lemma}
\newtheorem{problem}{Problem}
\newtheorem{definition}{Definition}
\newtheorem{remark}{Remark}
\usepackage{cleveref}
\crefname{lemma}{Lemma}{Lemmas}
\crefname{problem}{Problem}{Problems}
\crefname{definition}{Definition}{def}
\crefname{remark}{Remark}{rem}
\usepackage{xcolor}
\def\BibTeX{{\rm B\kern-.05em{\sc i\kern-.025em b}\kern-.08em
    T\kern-.1667em\lower.7ex\hbox{E}\kern-.125emX}}

\begin{document}

\title{Robust Controller Synthesis under Markovian Mode Switching with Periodic LTV Dynamics \\

\thanks{The authors are with the Department of AAE, Purdue University, West Lafayette, IN 47906 USA. (email: shriva15@purdue.edu, koguri@purdue.edu)}
}
\author{\IEEEauthorblockN{Shaurya Shrivastava}
\and
\IEEEauthorblockN{Kenshiro Oguri}
}
\maketitle
\begin{abstract}

In this work, we propose novel LMI-based controller synthesis frameworks for periodically time-varying Markov-jump linear systems. We first discuss the necessary conditions for mean square stability and derive Lyapunov-like conditions for stability assurance. To relax strict stability requirements, we introduce a new criterion that doesn't require the Lyapunov function to decrease at each time step. Further, we incorporate these stability theorems in LMI-based controller synthesis frameworks while considering two separate problems: minimizing a quadratic cost, and maximizing the region of attraction. Numerical simulations verify the controllers' stability and showcase its applicability to fault-tolerant control.
\end{abstract}

\section{Introduction} 


Markov Jump Linear Systems (MJLS) are an important kind of stochastic hybrid systems with applications in fault-tolerant control \cite{yang_fault-tolerant_2018,li_robust_2019} network control \cite{su_event-triggered_2021}, and macroeconomic models \cite{do_val_receding_1999}. In particular, they provide a powerful framework to design robust controllers with favorable properties such as stability or optimality even in the presence of partial/full actuator failures. MJLS are characterized by switching dynamics where the switching depends on an independent parameter, which forms a Markov chain. Therefore, the system dynamics comprise a set of modes, and the active mode at the time determines the dynamics. An excellent overview of the classical results in discrete-time MJLS literature can be found in \cite{costa_discrete-time_2005}. 

The stability characteristics of such systems are not trivial extensions of the classical linear systems case, and several non-intuitive findings arise. MJLS stability has been extensively studied in the literature \cite{costa_stability_1993} and the well-known Ricatti equations have also been extended to the MJLS framework \cite{costa_mean-square_1996}. To aid in controller synthesis, Linear Matrix Inequality(LMI)-based control design was proposed in \cite{el_ghaoui_robust_1996}. However, most of these results do not extend to time-varying dynamics and thus have limited practical applicability.


The wide applications of linear time-varying (LTV) systems has motivated several studies focused on their stability. For a single-mode LTV system, the Lyapunov stability criterion of finite decrease at every time step has been difficult to satisfy in several applications \cite{zhou_asymptotic_2017}. Therefore, considerable efforts have been made to relax this criterion while still establishing stability. Stable scalar functions have been utilized for discrete-time LTV systems in \cite{zhou_asymptotic_2017} and non-monotone Lyapunov functions have been studied in \cite{ahmadi_non-monotonic_2008}. 

Periodic LTV systems form a subclass of LTV systems and also serve as a bridge between LTI and LTV systems. \cite{bittanti_periodic_2008} summarizes some of the classical results on periodic system stability. For linear periodic systems, the Periodic Lyapunov Lemma (PLL) \cite{bittanti_extended_1985} utilizes time-varying quadratic Lyapunov functions to prove asymptotic stability. The PLL has also been used for LMI-based controller synthesis \cite{de_souza_lmi_2000} which allows for state and control constraints in the controller design process. Lyapunov-like conditions have also been derived for periodic nonlinear systems \cite{jiang_converse_2002,mazenc_strict_2003} which require finite decrease of a periodic Lyapunov function at every time step to guarantee stability. \cite{bohm_stability_2012} leveraged the periodic property to use a non-monotonic Lyapunov function which decreases over one period to guarantee stability. It was modified to formulate an equivalent convex control synthesis problem in \cite{athanasopoulos_stability_2014}. 

Although periodic LTV systems with Markov jumps have been studied before \cite{ma_h2_2013,dragan_stochastic_2005,hou_spectral_2016}, they have not dealt with linear quadratic control or explored any relaxations in the control synthesis framework. Therefore, in this work, we generalize the rich studies on periodic LTV system stability to periodically time-varying MJLS. The contributions of this work are threefold. First, we propose a new Lyapunov-like stability criterion for periodically time-varying MJLS in \Cref{thm4}. Second, we relax \Cref{thm4} by using the periodic property of the system in \Cref{thm_rel}. Third, we formulate LMI-based controller synthesis frameworks to design stable control laws for periodically time-varying MJLS while obeying convex state and control constraints.


\subsection{Notation:}$\mathbb{R}$, $\mathbb{N}$ and $\mathbb{Z}_a^b$ denote the set of real numbers, natural numbers and set of integers from $a$ to $b$, respectively. We define a linear space $\mathbb{M}^N$ with $N\in\mathbb{N}$ comprising a finite set of matrices $V\triangleq (V(1),\ldots V(N))$, where each $V(i)\in\mathbb{R}^{n\times n} $, $n\in\mathbb{N} \; \forall i\in\mathbb{Z}_1^N$.  $\{ V_r \}_{r\in\mathbb{N}}$ is a sequence of such finite set of matrices, and $V_k\in\mathbb{M}^N$ refers to the $k^{\mathrm{th}}$ element in the sequence. Therefore, $V_k(i)$ refers to the $i^{\mathrm{th}}$ matrix from the $k^{\mathrm{th}}$ element of the sequence. Similarly, $\mathcal{H}_k^i(\cdot)$ refers to the $i^{\mathrm{th}}$ operator from the $k^{\mathrm{th}}$ element of the sequence $\{\mathcal{H}_r(\cdot)\}_{r\in\mathbb{N}}$ whose each element is a finite set of operators $\mathcal{H}(\cdot)\triangleq(\mathcal{H}^1(\cdot),\ldots\mathcal{H}^N(\cdot))$.
 $\mathcal{H}^\star(\cdot)$ denotes the adjoint form of the operator $\mathcal{H}(\cdot)$, i.e. $\langle\mathcal{H}(X),Y\rangle=\langle X,\mathcal{H}^\star(Y)\rangle$, where $\langle\cdot,\cdot\rangle$ denotes the inner product. For a matrix $X$, $X^\top$, $\|X\|$, tr($X$) refers to its transpose, maximum singular value, and trace, respectively. For a set of matrices $V\in\mathbb{M}^N$, $\mathrm{blkdiag}\{V(i)\}$ denotes a block diagonal matrices comprising each matrix $V(i)$ of the finite set. For an operator  $\mathcal{F}$, $\sigma_m(\mathcal{F})$ indicates its spectral radius, given by $\sup_{\|X\|=1} \|\mathcal{F}(X)\|$. If an operator is defined for a finite set of matrices $V(i)\in\mathbb{M}^N$, it represents an operator on the matrix $\mathrm{blkdiag}(V(i))$. For a vector $x\in\mathbb{R}^n$, $\|x\|$ indicates the $l_2$-norm of the vector. $A\otimes B$ denotes the Kronecker product between matrices $A$ and $B$. $\ast$ is used to denote symmetrical elements in a hermitian matrix. $I_n$ denotes the identity matrix of size $n\times n$. A real-valued function $h(x):\mathbb{R}_+\to \mathbb{R}_+$ belongs to class-$\kappa$ if it is continuous, strictly increasing, and $h(0)=0$. If, in addition, $\lim_{t\rightarrow\infty}h(t)=\infty$, then it belongs to class $\kappa_\infty$.
\section{Background}\label{background}
\subsection{Periodic System Stability}
Before we discuss the classical results on periodic system stability, we first define a $T$-periodic matrix as a matrix obeying the following property:
\begin{equation}\label{periodicdef}
    P_k=P_{T+k}  \quad\quad\quad\forall k \in \mathbb{N}
\end{equation}
We now focus on periodic LTV systems of the form:
\begin{equation}\label{LTV}
    x_{k+1}=A_k x_k + B_k u_k
\end{equation}
where $A_k$ and $B_k$ are $T$-periodic matrices, $x_k \in \mathbb{R}^{n_x}$ is the state and $u_k \in \mathbb{R}^{n_u}$ is the control action at time $k$. The following theorem, called the PLL, establishes a Lyapunov stability criterion for such systems.
\begin{theorem}[Theorem 3 in \cite{bolzern_periodic_1988} and Proposition 3.5 in \cite{bittanti_periodic_2008}]\label{periodicLyapunovlemma}
  The LTV periodic system \eqref{LTV} is asymptotically stable using a linear feedback control $u_k=K_k x_k$ iff $\exists$  a $T$-periodic matrix $P_k\succ0$ such that 
\begin{equation}
    P_k - (A_k + B_k K_k)^\top P_{k+1} (A_k + B_k K_k) \succ 0, \forall k \in \mathbb{Z}^{T-1}_0
\end{equation}
\end{theorem}
While most works have focused on one-step decay of Lyapunov functions, some authors have also utilized the periodicity of the system to design less strict conditions which do not necessarily require the monotonic decrease of the Lyapunov function at every time step. Instead, they must decrease at one period, gives as $V_{k+T}(x_{k+T})-V_k(x_k)\leq0 \;\forall x,k\in\mathbb{Z}_0^{T-1}$.
\subsection{Markov Jump Linear Systems}
We consider a discrete-time MJLS defined on the probability space $(\Omega,\mathcal{M},\mathbb{P})$ given by:
\begin{equation}\label{MJLS}
    x_{k+1} = A_k(\omega_k)x_k + B_k(\omega_k)u_k
\end{equation}
We denote by $\mathcal{M}_k$ a sub-sigma algebra which represents the mode information until time $k$. The parameter $\omega_k$ corresponds to the active mode at time $k$, and it can take values from the set $\Omega = \{1,2,\cdots N\}$. The switching between these modes is governed by a homogeneous Markov chain with transition probability $p_{ij}$ where $p_{ij} = \mathbb{P}(r_{k+1}=j|r_k=i)$. We also define the transition probability matrix $\mathcal{P}$ whose ($i$,$j$)th entry is $p_{ij}$.

We focus on mode-dependent linear feedback controllers:
\begin{equation}\label{linearfeedback}
    u_k=K_k(i)x_k
\end{equation}
For such control laws, we can represent the dynamics propagation as $x_{k+1}=\phi_{k}(i)x_k$, where $\phi_k(i)=A_k(i)+B_k(i)K_k(i)$.
We now define some of the operators which are commonly used in the MJLS literature. For any set of matrices $V\in\mathbb{M}^N$,
\begin{equation}\label{operators}
    \begin{aligned}        \mathcal{E}^i(V)&\triangleq\sum_{j\in\Omega}p_{ij}V(j)\\
    \mathcal{T}^j_k(V) &\triangleq\sum_{i\in\Omega} p_{ij} \phi_k(i)V(i)\phi_k^\top(i)\\
    \mathcal{L}^i_k(V) &\triangleq \phi_k^\top(i)\mathcal{E}^i(V)\phi_k(i)
    \end{aligned}
\end{equation}
\begin{lemma}[Proposition 3.2 in \cite{costa_discrete-time_2005}]\label{operatoradjoint}
    ${\mathcal{T}^j_k}^\star(V)$ = $\mathcal{L}^i_k(V)$.
\end{lemma}
\begin{lemma}[Proposition 7.2.1 in \cite{kesavan_functional_2023}]\label{op_adj}
    If an operator is contractive, so is its adjoint.
\end{lemma}
We also define a finite set of linear operators $\mathcal{T}_k:\mathbb{M}^N\rightarrow\mathbb{M}^N$ and $\mathcal{L}_k:\mathbb{M}^N\rightarrow\mathbb{M}^N$ below:
\begin{equation}\label{diagoperators}
    \begin{aligned}
        \mathcal{T}_k(V)&\triangleq(\mathcal{T}^1_k(V),\ldots\mathcal{T}^N_k(V))\\
        \mathcal{L}_k(V)&\triangleq(\mathcal{L}^1_k(V),\ldots\mathcal{L}^N_k(V))
    \end{aligned}
\end{equation}
To study the covariance of a MJLS system as defined in \eqref{MJLS}, we also define the following:
\begin{equation}
    \begin{aligned}
\mathcal{X}_k(i)\triangleq&\mathbb{E}\left[x_kx_k^\top\mathbf{1}_{\omega_k=i}\right], \quad\mathcal{X}_k\triangleq(\mathcal{X}_k(1),\ldots\mathcal{X}_k(N))\\
X_k\triangleq&\mathbb{E}\left[x_kx_k^\top\right]=\sum_{i\in\Omega}\mathcal{X}_k(i)\\
    \end{aligned}
\end{equation}
where $\mathbf{1}_{\omega_k=i}$ is the indicator function. 
It has been shown \cite{costa_discrete-time_2005} that the covariance propagation is given by:
\begin{equation}\label{covprop}
    \mathcal{X}_{k+1}=\mathcal{T}_k(\mathcal{X}_k)
\end{equation}
We also define the operator $\mathcal{G}_k$ and $\mathcal{F}_k$ as follows:
\begin{equation}\label{gkdef}
    \begin{aligned}
        \mathcal{G}_k &\triangleq \mathcal{T}_{k-1}\circ\mathcal{T}_{k-2}\cdots\mathcal{T}_0\\
        \mathcal{F}_k &\triangleq \mathcal{L}_{0}\circ\mathcal{L}_{1}\cdots\mathcal{L}_{k-1}
    \end{aligned}
\end{equation}
\begin{definition}
    (Mean Square Stability) A MJLS is called Mean Square Stable (MSS) iff the following statement holds:
\begin{equation}
    \lim_{k\rightarrow\infty}\mathbb{E}\left[||x_k||^2\right] = 0
\end{equation}
\end{definition}
\subsection{Stability of Time-invariant MJLS}
For time-invariant MJLS, i.e. $\phi_k(\omega_k)=\phi(\omega_k)$, the criteria for MSS is given in \Cref{MSSLTI}:
\begin{theorem}[Theorem 3.9 in \cite{costa_discrete-time_2005}]\label{MSSLTI}
     A time-invariant MJLS is MSS iff $\|\mathcal{A}\|<1$, where $\mathcal{A}=(\mathcal{P}\otimes I_{n_x^2}) \mathrm{blkdiag}(\phi(i)\otimes \phi(i))$. 
\end{theorem}
\begin{remark}
    It is important to note that the stability of the individual modes does not directly imply MSS and vice-versa. For example, consider the following system:
\begin{equation*}
\begin{aligned}
        A_1 = \begin{bmatrix}
-0.5 & 2.0\\
                -0.5  &0.5
\end{bmatrix} &\; A_2 = \begin{bmatrix}
-0.5 & 0.1\\
                1  &0.3
\end{bmatrix}
\mathcal{P} = \begin{bmatrix}
0.6& 0.4\\
0.5&0.5
\end{bmatrix}
\end{aligned}
\end{equation*}
Here, $A_1$ and $A_2$ are stable, but the MJLS is not MSS.
\end{remark}
There also exists a Lyapunov-like condition to guarantee MSS of MJLS as given in \Cref{LyapunovLTI}:
\begin{lemma}[Theorem 3.19 in \cite{costa_discrete-time_2005}]\label{LyapunovLTI}
    A time-invariant MJLS is MSS iff $\exists$ a finite set of positive definite matrices $P = [P(1)\cdots P(N)] \in \mathbb{M}^N$ such that
\begin{equation}
    \begin{aligned}
        P - \mathcal{T}(P) \succ 0
    \end{aligned}
\end{equation}
where $\mathcal{T}(\cdot)$ is the time-invariant analogue of $\mathcal{T}_k$ defined in \eqref{diagoperators}. 
\end{lemma} 
\subsection{Stability of Periodically Time-varying MJLS}
 The standard LTV MJLS \eqref{MJLS} is called LTV periodic MJLS (LTVPM) if each individual dynamics matrix $\phi_k(i)$ is $T$-periodic. Their stability criterion are given below:
 \begin{theorem}[Lemma 1 in \cite{hou_spectral_2016}] 
    A LTVPM is MSS iff the linear operator $\mathcal{G}_T$, as defined in \eqref{gkdef}, is contractive.
\end{theorem}
\section{Stability Criterion for LTVPM}

\subsection{Periodic Lyapunov Lemma for MJLS}\label{PLL_MJLS_sec}
We first generalize \Cref{periodicLyapunovlemma} to Markov jump systems:
\begin{theorem}\label{thm4}
A LTVPM is MSS iff $\exists$ a T-periodic finite set of matrices $P_k\in\mathbb{M}^N$ such that $P_k(i)\succ 0 \quad \forall i \in \Omega$, and
\begin{equation}\label{pll_MJLS}
    \begin{aligned}
        P_k - \mathcal{L}_k(P_{k+1}) \succ 0 \quad\quad  \forall k\in\mathbb{Z}_0^{T-1}
    \end{aligned}
\end{equation}
\end{theorem}  
\begin{proof}
To first see that \eqref{pll_MJLS} implies MSS, we first iteratively use \eqref{pll_MJLS} over one period to get \eqref{thm3p1}:
\begin{equation}\label{thm3p1}
    \begin{aligned}
        P_0 \succ & \mathcal{L}_0(P_{1})
        \succ  \mathcal{L}_0( \mathcal{L}_1(P_{2}))\succ\cdots\\
        &\cdots \succ \left[\mathcal{L}_0\circ \mathcal{L}_1 \cdots \mathcal{L}_{T-1}\right](P_0) = \mathcal{F}_T(P_0)\\
    \end{aligned}
\end{equation}
Therefore, this implies that the linear operator $\mathcal{F}_T$ is contractive, or its spectral radius is less than one. Now, from \cref{operatoradjoint}, we know that $\mathcal{L}_k^\star=\mathcal{T}_k$, which implies $\mathcal{G}_T=\mathcal{T}_{T-1}\circ\mathcal{T}_{T-2}\cdots\mathcal{T}_0$ is also contractive using \Cref{op_adj}. 

Now, using \eqref{gkdef}, we can consider covariance propagation for one-time period $T$  as $\mathcal{X}_T=\mathcal{G}_T(\mathcal{X}_0)$. For a LTVPM, the dynamics are $T$-periodic, which implies that the dynamics for $\mathcal{X}_T=\mathcal{G}_T(\mathcal{X}_0)$ is also $T$-periodic. Therefore, $\lim_{k\rightarrow\infty}X_k$ exists and will converge to zero iff $\sigma_m(\mathcal{G}_T)<1$, or, in other words, the operator $\mathcal{G}_T$ is contractive, which implies MSS.

For the \textit{only if} part, we first assume that the system is MSS. We then define the state transition matrix for a given mode history $\{\omega_{k_0}\ldots\omega_{k-1}\}$ as:
\begin{equation}
    \Phi(k,k_0) = \phi_{k-1}(\omega_{k-1})\ldots\phi_{k_0}(\omega_{k_0})
\end{equation}
The MSS assumption of the MJLS implies $\|\Phi(T,0)\|<1$. By definition, $\Phi(k_0,k_0)\triangleq I$. For the same mode history $\{\omega_{k_0}\ldots\omega_{k_1}\ldots\omega_{k_2-1}\}$, the following relation stands: $\Phi(k_2,k_0)=\Phi(k_2,k_1)\Phi(k_1,k_0)\;\;\:\forall k_2\geq k_1 \geq k_0$.

Now, let us define 
\begin{equation}\label{Pkdef}
    P_k(i) \triangleq \mathbb{E}\left[ \sum_{t=0}^\infty \Phi(k+t,k)^\top\Phi(k+t,k) \mid \omega_k=i \right]
\end{equation}
We first show that this matrix infinite series exists. For $r\in\mathbb{N}$,
\begin{equation}
    \begin{aligned}
         \mathbb{E}\left[ \sum_{t=0}^{rT-1} \Phi(k+t,k)^\top\Phi(k+t,k) \mid \omega_k=i \right]         \\
        = \mathbb{E}\left[\sum_{j=0}^{r-1}(\Phi(k+T,k)^\top)^j\Psi(k)(\Phi(k+T,k))^j \right]   
    \end{aligned}
\end{equation}
where $\Psi(k)=\sum_{t=0}^{T-1} \Phi(k+t,k)^\top\Phi(k+t,k)$. If we take the limit $r\rightarrow\infty$, $\|\Phi(T,0)\|<1$ guarantees the existence of $P_k(i)$ defined in \eqref{Pkdef}. Now, RHS of \eqref{Pkdef} can be written as:
\begin{equation}\label{pkdef2}
    \begin{aligned}
        \mathbb{E}\left[ \sum_{t=0}^\infty \Phi(k+1+t,k)^\top\Phi(k+1+t,k) \mid \omega_k=i \right]   +I
    \end{aligned}
\end{equation}
Simplifying the first term, we get
\begin{equation}\label{pk1}
\begin{aligned}
    \mathbb{E}\left[ \sum_{t=0}^\infty \Phi(k+1+t,k+1)^\top\Phi(k+1+t,k+1) \mid \omega_k=i \right] \\
    =\phi_k^\top(i)\left(\sum_j p_{ij} P_{k+1}(j)\right)\phi_k(i)\;    = \phi_k^\top(i)\mathcal{E}^i(P_{k+1}) \phi_k(i)
\end{aligned}
\end{equation}
Using \eqref{pk1} and \eqref{pkdef2} in \eqref{Pkdef}, we get 
\begin{equation}\label{thm4finaleq}
    P_k(i) = \phi_k^\top(i)\mathcal{E}^i(P_{k+1}) \phi_k(i) + I
\end{equation}
which implies $P_k(i) - \mathcal{L}_k^i(P_{k+1})=I\succ 0$.
\end{proof}
Next, we utilize \Cref{thm4} to find an upper bound for a quadratic cost metric, detailed in \cref{lemma_performance}:
\begin{lemma}\label{lemma_performance}
    (Performance analysis lemma) Suppose there exist T-periodic mode-dependent positive definite matrices $M_k(i)$ and $P_k(i)$, and a positive scalar constant $\beta>0$ such that 
    \begin{equation}\label{per_lemma}
        \begin{aligned}
            P_k(i) \succ \mathcal{L}_k^i(P_{k+1}) + \frac{M_k(i)}{\beta}
        \end{aligned}
    \end{equation}
    $\forall k \in \mathbb{Z}_0^{T-1}$ and $i\in\Omega$, then the corresponding MJLS is MSS, with $\mathbb{E}  \left[\sum_{k=0}^\infty x_k^\top M_k(i)x_k\right] \leq \beta x_0^\top P_0(i) x_0$.
\end{lemma}
\begin{proof}
\eqref{per_lemma} already satisfies \eqref{pll_MJLS} since we assume $M_k(i)\succ 0$ and $\beta>0$, which implies the corresponding MJLS is MSS. Now, pre and post-multiplying  \eqref{per_lemma} with $x_k^\top$ and $x_k$ respectively, and using the $x_{k+1}=\phi_k(i)x_k$, we have:
\begin{equation}
    x_k^\top P_k(i) x_k \geq x_{k+1}^\top\mathcal{E}^i(P_{k+1})x_{k+1} + \frac{x_k^\top M_k(i)x_k}{\beta}
    \label{pl4}
\end{equation}

Now, considering the conditional expectation term $\mathbb{E}\left[x_{k+1}^\top P_{k+1}x_{k+1}\lvert \mathcal{M}_k \right]$, $x_{k+1}$ is a deterministic quantity as the mode at time step $k$ is given. Therefore, we can write $ x_{k+1}^\top\mathcal{E}^i(P_{k+1})x_{k+1} = \mathbb{E}\left[x_{k+1}^\top P_{k+1}x_{k+1}
\lvert\mathcal{M}_k\right]$.  Taking expectation of both sides in \eqref{pl4} and using law of total expectation, we get:
\begin{equation}
    \mathbb{E}\left[ x_k^\top P_k(i) x_k\right] \geq \mathbb{E} \left[ x_{k+1}^\top P_{k+1}x_{k+1}\right] + \frac{\mathbb{E}\left[ C_k(i)\right]}{\beta}
     \label{pl5}
\end{equation}
where $C_k(i)=x_k^\top M_k(i) x_k$. Adding \eqref{pl5} from $k=0$ to $T-1$, we have the following inequality:
\begin{equation}
    \mathbb{E}\left[ x_0^\top P_0(i) x_0\right] \geq \mathbb{E} \left[ x_{T}^\top P_{T}x_{T} \right] + \frac{ \sum_{k=0}^{T-1}\mathbb{E}\left[ C_k(i)\right]}{\beta}
    \label{pl5a}
\end{equation}
Adding \eqref{pl5a} for $r$-time periods, we get:
\begin{equation}\nonumber
    \mathbb{E}\left[ x_0^\top P_0(i) x_0\right] \geq \mathbb{E} \left[ x_{rT}^\top P_{rT}x_{rT}\right] + \frac{ \sum_{k=0}^{rT-1}\mathbb{E}\left[ C_k(i)\right]}{\beta}
    \label{pl6}
\end{equation}
Taking the limit $r\rightarrow\infty$ implies $\mathbb{E}\{  x_{rT}^\top P_{rT}x_{rT}\} \rightarrow 0$ from MSS, so we have $\sum_{k=0}^{\infty} \mathbb{E}\left[ C_k(i)\right] \leq \beta x_0^\top P_0(i) x_0$.
\end{proof}
\subsection{Relaxed Stability Criterion for LTVPM}\label{relaxed_PLL_sec}

\Cref{thm4} requires the Lyapunov function to decrease at every time step, which can be quite conservative in practice. We now relax this condition by proving that a finite decrease of the Lyapunov function over one period is sufficient for MSS.


\begin{theorem} \label{thm_rel}
    A LTVPM is MSS iff $\exists$ a T-periodic finite set of matrices $P_k\in \mathbb{M}^N$ and a finite sequence of positive real numbers $\{\nu_k\}_0^{T-1}$ such that $P_k(i)\succ 0\quad \forall i \in\Omega$, and 
    \begin{subequations}
        \begin{equation}\label{pll_rel_MJLS}
            \nu_k P_k - \mathcal{L}_k(P_{k+1}) \succeq 0 \quad\quad  \forall k\in\mathbb{Z}_0^{T-1}
        \end{equation}
        \begin{equation}\label{nu_cond}
            \hspace{-18pt}\prod_{k=0}^{T-1} \nu_k <1
        \end{equation}
    \end{subequations}
\end{theorem}
\begin{proof}
    We first iteratively use \eqref{pll_rel_MJLS} over one period to get:
    \begin{equation}\label{rel_p1}
    \begin{aligned}
        P_0 \succeq \frac{\mathcal{L}_0(P_{1})}{\nu_0}
        \succeq\cdots\succeq \frac{\left[\mathcal{L}_0\circ \mathcal{L}_1 \cdots \mathcal{L}_{T-1}\right](P_0)}{\prod_{k=0}^{T-1} \nu_k}\succ\mathcal{F}_T(P_0)
    \end{aligned}
\end{equation}
where the last step follows from \eqref{nu_cond}. 
\eqref{rel_p1} implies the operator $\mathcal{F}_t$ is contractive which proves MSS, as shown previously in the proof of \Cref{thm4}.

For the \textit{only if} part, we first assume that the system is MSS. From \Cref{thm4}, we know that this is equivalent to the existence of a T-periodic finite set of positive definite matrices $P_k\in\mathbb{M}^N$ such that \eqref{pll_MJLS} is true. \eqref{pll_MJLS} implies that if we choose each $\nu_k<1$, $\exists$ a finite sequence $\{\nu_k\}_0^{T-1}$ which obey \eqref{pll_rel_MJLS}. Also, since each $\nu_k<1$, \eqref{nu_cond} is also satisfied.
\end{proof}
\section{LMI-based Controller Design}\label{sec_framework}
We consider a general optimal control problem of the form:
\begin{equation}
    \begin{aligned}
        \min_{K_k(i)} J \\
        \text{s.t.}  \forall i \in \Omega,k\in\mathbb{N},\quad
            {{x}}_{k+1} &= A_k(i) {x}_k + B_k(i) {u}_k,\\
            A_k(i)&=A_{k+T}(i)\\
            B_k(i)&=B_{k+T}(i)\\
             p_{ij} &= \mathbb{P}(r_{k+1}=j|r_k=i)\\
             u_k(i)&=K_k(i)x_k(i)\\
            \|u_k(i)\| &\leq u_m(i) \quad \quad  w.p.\;\; 1 \\
            x_k(i)^\top W_k(i) x_k(i) &\leq 1\quad \quad \quad \quad w.p.\;\; 1  
    \end{aligned}
    \label{ocp_MJLS}
\end{equation}
We consider two optimal control problems (OCP) which aim to stabilize LTVPM. The two problems differ in the initial conditions and objective functions: Problem 1 minimizes a quadratic cost, with the initial state contained in a convex hull, while Problem 2 maximizes the region of attraction for the system. The corresponding OCP is given below:
\begin{problem}\label{prob1}
    Solve \eqref{ocp_MJLS} with  $x_0\in conv\{x_{0,v}\},\forall v \in \mathbb{Z}_1^l$ and $J = \mathbb{E}\left[\sum_{k=0}^{\infty} x_k^\top Q(i) x_k + u_k^\top R(i) u_k\right]$.
\end{problem}
\begin{problem}\label{prob2}
    Solve \eqref{ocp_MJLS} for $J = -tr(\mathcal{R})$ where R is the largest set such that if $x_0\in \mathcal{R}$, then $\lim_{k\rightarrow\infty}\mathbb{E}\left[||x_k||^2\right] \rightarrow 0$.
\end{problem}

\Cref{LMI1_sec} and \Cref{LMI2_sec} formulate Problem 1 and 2 in a semi-definite program (SDP):
\subsection{Problem 1 control synthesis SDP}\label{LMI1_sec}
\begin{theorem}\label{LMI1}
The optimal controller for \Cref{prob1} is given by $K_k(i)=\bar{Y}_k(i)\bar{S}_k^{-1}(i)$, where $\bar{Y}_k(i)$ and $\bar{S}_k^{-1}(i)$ are the optimal value of the following SDP:
\begin{equation*}
\begin{aligned}
\min_{\beta,Y_k(i),S_k(i)} \beta \\
\end{aligned}
\end{equation*}
\begin{subequations}
   \begin{equation}
   \text{s.t.}\quad     \begin{bsmallmatrix}
            1 & \ast \\
x_{0, v} & S_0(i)
        \end{bsmallmatrix}
         \geq 0, \quad \forall v=\mathbb{Z}_1^l
         \label{t1}
    \end{equation}
    \begin{equation}
        \begin{bsmallmatrix}
            S_k(i) & \ast & \ast & \ast  \\
    \ell_{i}^{\top}\left(A_k(i) S_k(i)+B_k(i) Y_k(i)\right) & \operatorname{blkdiag}\left\{S_{k+1}\right\} &  \ast & \ast\\
Q(i)S_k(i) & 0 & \beta I & \ast\\
R(i)Y_k(i) &0 &0 & \beta I
        \end{bsmallmatrix}\geq 0
        \label{t2}
    \end{equation}
    \begin{equation}
        \begin{bsmallmatrix}
            S_k(i) & \ast \\
A_k(i) S_k(i)+B_k(i) Y_k(i) & S_{k+1}(j)
        \end{bsmallmatrix}\geq 0,  \quad  \forall j \in \Omega
        \label{t3}
    \end{equation}
    \begin{equation}
\begin{bsmallmatrix}
u_m(i)^2 I &  \ast \\
Y_k(i)^{\top} & S_k(i)
\end{bsmallmatrix} \geq 0
    \label{t4}
\end{equation}
\begin{equation}
        \begin{bsmallmatrix}I- H_k(i) S_k(i) H_k^\top(i)\end{bsmallmatrix} \geq 0
    \label{t5}
\end{equation}
\end{subequations}





$\forall$ $i\in\Omega, k=\mathbb{Z}_0^{T-1}$, where $\ell_i=[\sqrt{p_{i1}}I\ldots\sqrt{p_{iN}}I]$, $H_k(i)=W_k^{1/2}(i)$, 
and $S_k,Y_k\in\mathbb{M}^N$. $\beta$, $S_k(i)$, $Y_k(i)$ are the decision variables in this SDP and all other matrices are defined from the parameters given in the problem \eqref{ocp_MJLS}.
\end{theorem} 
\begin{proof}

Using Schur complement on \eqref{t1}, we get:
\begin{equation}
    x_{0, v}^{\top}S_0^{-1}(i)x_{0, v} \leq 1
    \label{p1}
\end{equation}
Using Schur complement again on \eqref{t2}, we have:
\begin{equation}
\begin{aligned}
        S_k(i)- \frac{1}{\beta}(S_k(i)Q_k(i)S_k(i) + Y_k(i)R_k(i)Y_k(i))-\\(\Gamma_k(i))^\top(\sum_{j=1}^{N}p_{ij}S^{-1}_{k+1}(j))\Gamma_k(i)\geq0
\end{aligned}
    \label{p2}
\end{equation}
where $\Gamma_k(i)=A_k(i) S_k(i)+B_k(i) Y_k(i)$.
Pre and post-multiplying \eqref{p2} with $S_k^{-1}(i)$,
\begin{equation}\nonumber
    \begin{aligned}
        P_k(i) \geq\phi_k^\top(i)\mathcal{E}^i(P_{k+1})\phi_k(i)   + \frac{Q_k(i)+K_k^\top(i)R_k(i)K_k(i)}{\beta}
    \end{aligned}
    \label{p3}
\end{equation}
where $P_k(i)=S_k^{-1}(i)$ and $K_k(i) = Y_k(i) S_k^{-1}(i)$. Now, using \Cref{lemma_performance}, the system is MSS, with the upper bound on $J$:
\begin{equation}
     J(K) = \sum_{k=0}^{\infty} \mathbb{E}\left[ C_k(i)\right] \leq \beta x_0^\top P_0(i) x_0
     \label{p6}
\end{equation}
where $C_k(i)=x^\top_kQ_k(i)x_k+u_k(i)^\top R_k(i) u_k(i)$, using \eqref{linearfeedback}.
Any initial state $x_0$ in the convex hull of $x_{0,v}$ can be expressed as $\sum_v \alpha_v x_{0,v}$ where $\sum_v \alpha_v=1$,\;$\alpha_v\geq0$. Therefore, $\begin{bsmallmatrix}
    1 & \ast \\
x_{0} & S_i
\end{bsmallmatrix}=\sum_v \alpha_v\begin{bsmallmatrix}
    1 & \ast \\
x_{0, v} & S_i
\end{bsmallmatrix}\geq 0$
which implies $x_0^\top P_0(i)x_0\leq1\;\;\forall x_0\in conv\{x_{0,1}\cdots x_{0,l}\}$ from \eqref{p1}. This simplifies \eqref{p6} to $J(K)\leq \beta$.
Now, using Schur's complement on \eqref{t3}, 
\begin{equation}
    \begin{aligned}
        S_k(i) \geq (\Gamma_k(i))^\top S_{k+1}^{-1}(j)\Gamma_k(i)
    \end{aligned}
    \label{p3a}
\end{equation}
Pre and post-multiplying with $S_k(i)$ and using $P_k(i)=S_k^{-1}(i)$
\begin{equation}
    \begin{aligned}
        P_k(i) &\geq \phi_k^\top(i)P_{k+1}(j)\phi_k(i)
    \end{aligned}
    \label{pi7}
\end{equation}
Pre and post-multiplying \eqref{pi7} with $x_k^\top$ and $x_k$ respectively,
\begin{equation}
    x_k^\top P_k(i) x_k \geq x_{k+1}^\top P_{k+1}(j) x_{k+1} 
    \label{p7}
\end{equation}
Now, enforcing \eqref{t3} $\forall j,i\in \Omega$ implies that that \eqref{p7} is true under any switching to mode $j$ from any given mode $i$. This is crucial to ensure that the state and control constraints are obeyed with probability one. Combining \eqref{p1} with \eqref{p7} gives:
\begin{equation}
    \begin{aligned}
        x_{k+1}^\top P_{k+1}(j) x_{k+1} \leq x_k^\top P_k(i) x_k \leq x_0^\top P_0(i_0) x_0 \leq 1
    \end{aligned}
    \label{p7a}
\end{equation}
$ \forall k \in \mathbb{Z}_0^{T-1}$, where $i_0$ represents the initial active mode. From the last time step, we get $x_{T-1}^\top P_{T-1}(i) x_{T-1} \geq x_{T}^\top P_T(j) x_{T}$. Repeating the process by utilizing the periodicity of the Lyapunov matrices, we have $x_k^\top P_k(i)x_k\leq1 \;\;\forall k \in \mathbb{Z}$.

\noindent Now, using Schur's complement on \eqref{t4},
\begin{equation}
    Y_k(i)S_{k}^{-1}(i)Y_k^\top(i) \leq u_m^2I
\end{equation}
Taking the norm of this inequality (as both sides are PD,)
\begin{equation}
    \|Y_k(i) S_k^{-1/2}(i)\|^2 \leq u_m^2(i)
\end{equation}
Multiplying both sides with $\|S_k^{-1/2}(i) x_k\|^2$,
\begin{equation}
    \begin{aligned}
        &\|Y_k(i) S_k^{-1/2}(i)\|^2\|S_k^{-1/2}(i) x_k\|^2 \leq u_m^2(i)\|S_k^{-1/2}(i) x_k\|^2\\
        &\implies\|Y_k(i) S_k^{-1}(i) x_k\|^2 \leq u_m^2(i)\|S_k^{-1/2}(i)x_k\|^2\hspace{-5pt}
    \end{aligned}
    \label{p8}
\end{equation}
From \eqref{p7a}, we know $|S_k^{-1/2}(i)x_k\|^2 \leq 1$. Using \eqref{linearfeedback}, \eqref{p8} simplifies to:
\begin{equation}
    \begin{aligned}
           \|u_k(i)\| & \leq u_m(i)
    \end{aligned}
\end{equation}
Now, rearranging \eqref{t5}, we get
\begin{equation}
    H_k(i) S_k(i) H_k^\top(i) \leq I
    \label{p9}
\end{equation}
Taking the norm of both sides of \eqref{p9}, multiplying both sides with $\|S_k^{-1/2}(i) x_k\|^2$, and using \eqref{p7a}, we get:
\begin{equation}\nonumber
    \begin{aligned}
        \|H_k(i)S_k^{1/2}(i)\|^2\|S_k^{-1/2}(i) x_k\|^2 \leq \|S_k^{-1/2}(i) x_k\|^2 \leq 1\\
    \end{aligned}
\end{equation} 
\begin{equation*}
    \implies\|H_k(i) x_k\|^2 \leq1\Longleftrightarrow x_k^\top(i)W_k(i)x_k(i) \leq 1 \qedhere
\end{equation*}
\end{proof}
\subsection{Problem 2 control synthesis SDP}\label{LMI2_sec}
To incorporate Theorem \ref{thm_rel} into controller design, we focus on maximizing region of attraction problems. We utilize the useful property of Lyapunov sublevel sets which allows them to also serve as invariant sets. In particular, if we maximize $\mathcal{R}:=\{x:x_0^\top P_0(i)x_0\leq 1\}$, it will lead to the largest possible region of attraction while guaranteeing MSS. To enforce control and state constraints with probability one, we limit $\nu_k$ to the set $(0,1]$. This is still a relaxation compared to \Cref{thm4} as we allow the Lyapunov function to not decrease. Based on this idea, we formulate Problem 2 in a SDP in \Cref{LMI_rel}. We choose to maximize the trace of $\mathcal{R}$ which is equivalent to maximizing its volume. Other metrics for the size of $\mathcal{R}$ are discussed in \cite{boyd_linear_1994}.
\begin{theorem}\label{LMI_rel}
The optimal controller for \Cref{prob2} is given by $K_k(i)=\bar{Y}_k(i)\bar{S}_k^{-1}(i)$, where $\bar{Y}_k(i)$ and $\bar{S}_k^{-1}(i)$ are the optimal value of the following SDP:
\begin{equation*}
\begin{aligned}
\min_{Y_k(i),S_k(i)} -\text{tr}(\mathbb{E}_\rho[S_0]) \\
\end{aligned}
\end{equation*}
\begin{subequations}
    \begin{equation}\label{tr2}
        \text{s.t.}\quad\begin{bsmallmatrix}
    \nu_kS_k(i) & \ast  \\
    \ell_{i}^{\top}\left(A_k(i) S_k(i)+B_k(i) Y_k(i)\right) & \operatorname{blkdiag}\left\{S_{k+1}\right\}
\end{bsmallmatrix}\geq 0
\end{equation}
\begin{equation}\label{tr3}
{\begin{bsmallmatrix}
S_k(i) & \ast \\
A_k(i) S_k(i)+B_k(i) Y_k(i) & S_{k+1}(j)
\end{bsmallmatrix} \geq 0}  \quad  \forall j \in \Omega
\end{equation}
\begin{equation}
\begin{bsmallmatrix}
u_m(i)^2 I &  \ast \\
Y_k(i)^{\top} & S_k(i)
\end{bsmallmatrix} \geq 0
    \label{tr4}
\end{equation}
\begin{equation}
        \begin{bsmallmatrix}I- H_k(i) S_k(i) H_k^\top(i)\end{bsmallmatrix} \geq 0
    \label{tr5}
\end{equation}
\end{subequations}

$\forall$ $i\in\Omega, k=\mathbb{Z}_0^{T-1}$, where $\nu_k>0$ is the $k^\mathrm{th}$ element of a given finite sequence $\{\nu_k\}_0^{T-1}$ such that $\prod_{k=0}^{T-1} \nu_k <1$, and $\rho$ is the initial mode distribution.
\end{theorem}
\begin{proof}
Using Schur complement on \eqref{tr2}, we have:
\begin{equation}
\begin{aligned}
        \nu_k S_k(i) - (\Gamma_k(i))^\top(\sum_{j=1}^{N}p_{ij}S^{-1}_{k+1}(j))\Gamma_k(i) \geq 0 
\end{aligned}
    \label{pr2}
\end{equation}
Set $P_k(i)=S_k^{-1}(i)$ and $K_k(i) = Y_k(i) S_k^{-1}(i)$ $\forall k,i \in \Omega$. Pre and post-multiplying \eqref{pr2} with $S_k^{-1}(i)$, we have:
\begin{equation}
    \begin{aligned}
        \nu_k P_k(i) \geq\phi_k^\top(i)\mathcal{E}^i(P_{k+1})\phi_k(i)
    \end{aligned}
    \label{pr3}
\end{equation}
From \Cref{thm_rel}, we know that \eqref{pr3} implies MSS.
State and control constraints can be proved using the same approach as in Theorem \ref{LMI1} for \eqref{t3},\eqref{t4} and \eqref{t5}. 
We also note that \eqref{tr3} enforces that the Lyapunov function \emph{must not increase} at the next time step, across all modes. This implies that if the state is initially in the set $\mathcal{R}:=\{x_0:x_0^\top S^{-1}_0(i)x_0\leq 1\}$, it will stay in that set for all future times, with probability one, serving as an invariant set.
\end{proof}
\section{Numerical Simulation}\label{sims_sec}
We study a LTVPM with the following dynamics:
\begin{equation}
    \begin{aligned}
        &A_k(1)=\begin{bsmallmatrix}
    -0.5 & 2\\
        -0.4 & 0.8\mathrm{sin}(0.2\pi k)
\end{bsmallmatrix},
A_k(2)=\begin{bsmallmatrix}
    0.5\mathrm{cos}(0.2\pi k) & 0.5\\
        0.8 & 0.5
\end{bsmallmatrix}\\
&B_k(1) =\begin{bsmallmatrix}
    1\\1
\end{bsmallmatrix}, \quad B_k(2)=\begin{bsmallmatrix}
    0\\0
\end{bsmallmatrix},\quad
\mathbb{P} = \begin{bsmallmatrix}
    0.8&0.2\\
        0.9&0.1
\end{bsmallmatrix},\quad T=10
    \end{aligned}
    \label{mjls_eg}
\end{equation}
We first discuss the stability of this LTVPM. The one-period state transition matrices for each mode $\Phi_1=\prod_{k=0}^{T-1}A_k(1)$ and $\Phi_2=\prod_{k=0}^{T-1}A_k(2)$ have all its eigenvalues within the unit circle ($\|\Phi_1\|=0.23$ and $\|\Phi_2\|=0.55$) i.e. the individual modes are stabilizing. However, under the Markovian jump dynamics defined by $\mathbb{P}$ in \eqref{mjls_eg}, the one-period operator $\sigma_m(\mathcal{G}_T)=1.255>1$, implying that the LTVPM is unstable. Similar systems with $B_2(k)=0$ can also be used to model control failure for a system, demonstrating the applicability of our work to fault-tolerant control design.

\subsection{Problem 1}

For problem 1, the cost weight matrices are set as: $Q(1)=Q(2)=0, R(1)=R(2)=1$. We use \Cref{LMI1} to synthesize a robust optimal controller for this system. We impose an upper bound on the control norm $u_m(i)=125,\forall i \in\Omega$. The initial state is sampled randomly from $conv\{[\alpha,\alpha],[\alpha,-\alpha],[-\alpha,-\alpha],[-\alpha,\alpha]\}$ where $\alpha=100$, and the initial mode is randomly sampled from a uniform distribution. Solving the LMIs takes $672$ ms on a Macbook Pro M2 Pro using Julia programming language with JuMP interface and Clarabel\cite{goulart_clarabel_2024} solver. After solving for the optimal controller, the one-period STM $\Phi$ was computed and the controlled MJLS was found to be stable ($\sigma_m(\Phi)=0.011$.) 

\Cref{fig:c_xhist} shows the state trajectories for 1000 Monte Carlo simulations with the optimal controller, while \Cref{fig:uc_xhist} plots the same for the uncontrolled case. Covariance was calculated analytically, and $3\sigma$ bounds shown in the plots provide a reasonable estimate for an upper bound of the state trajectories.
\subsection{Problem 2}
For Problem 2, we set the Lyapunov function to not decrease for all but the $4^\mathrm{th}$ time step, whereas a decrease of 0.9 times its previous step is enforced for $k=4$. This implies $\nu_k=0\forall k \in \mathbb{Z}_0^{10}/4$ and $\nu_4=0.9$. We set an upper bound on the control norm and state norm as $u_m(i)=125$ and 
$W_k(i)=I_2/\delta^2$ respectively, $\forall i \in \Omega,$ where $\delta=250$. For Monte Carlo simulations, the initial condition is sampled uniformly from the region of attraction as found by the optimizer. All trajectories are found to stabilize with theoretical guarantees ($\sigma_m(\mathcal{G}_T)=0.015$). \Cref{fig:rel_norm_xhist} and \Cref{fig:rel_uc_xhist} show the state norm history and control norm history for 1000 Monte Carlo simulations respectively. Both the state and control constraints are obeyed with probability one. 
\begin{figure}[htbp!]
     \centering
     \begin{subfigure}{0.24\textwidth}
         \centering
         \includegraphics[width=\textwidth]{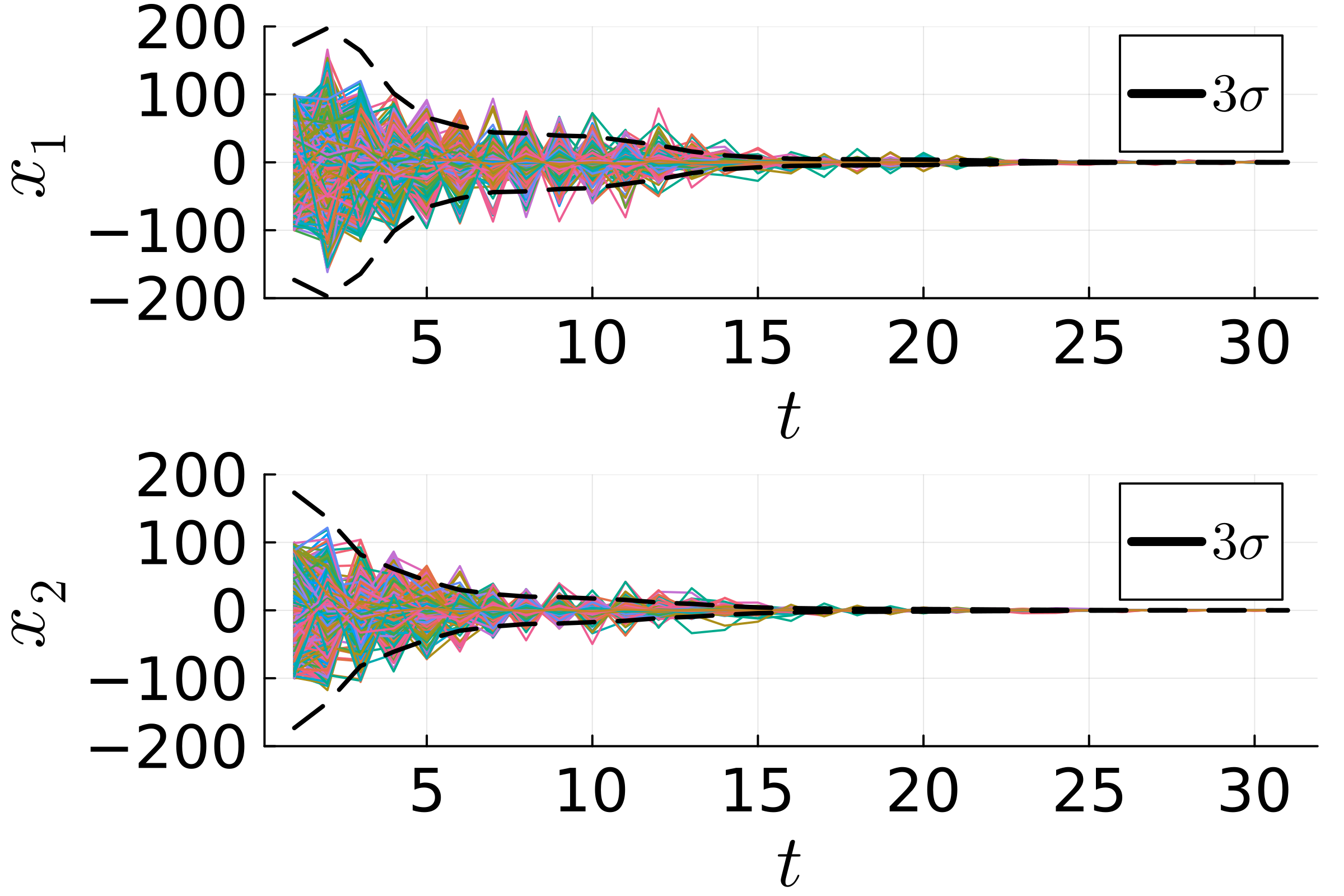}
         \caption{controlled case}
         \label{fig:c_xhist}
     \end{subfigure}
     \hfill
     \begin{subfigure}{0.24\textwidth}
         \centering
         \includegraphics[width=\textwidth]{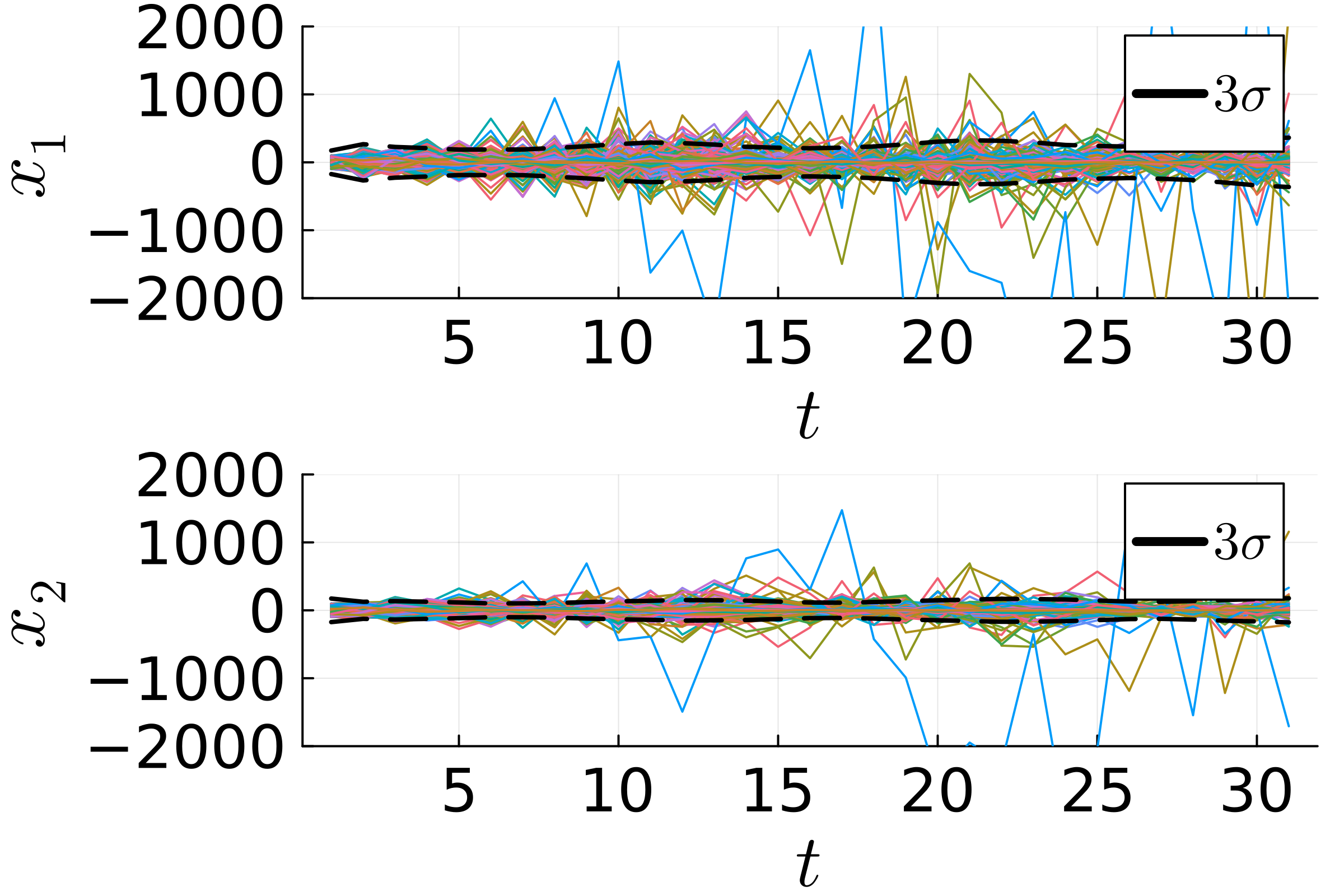}
         \caption{uncontrolled case}
         \label{fig:uc_xhist}
     \end{subfigure}
     \hfill
     \caption{Problem 1 Monte Carlo simulations: State trajectories}
\end{figure}
\begin{figure}[htbp!]
     \centering
     \begin{subfigure}{0.24\textwidth}
         \centering
         \includegraphics[width=\textwidth]{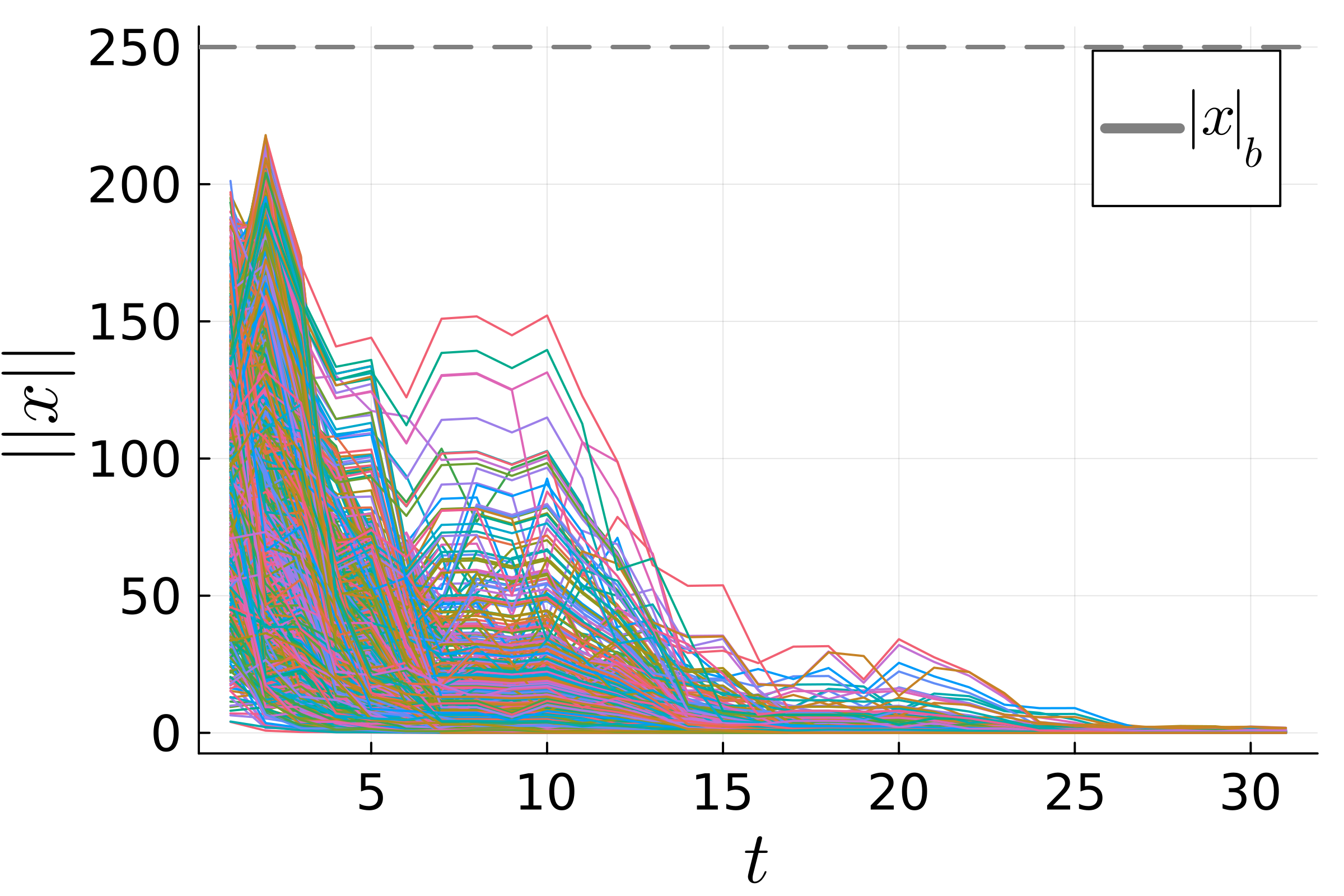}
         \caption{State norm}
         \label{fig:rel_norm_xhist}
     \end{subfigure}
     \hfill
     \begin{subfigure}{0.24\textwidth}
         \centering
         \includegraphics[width=\textwidth]{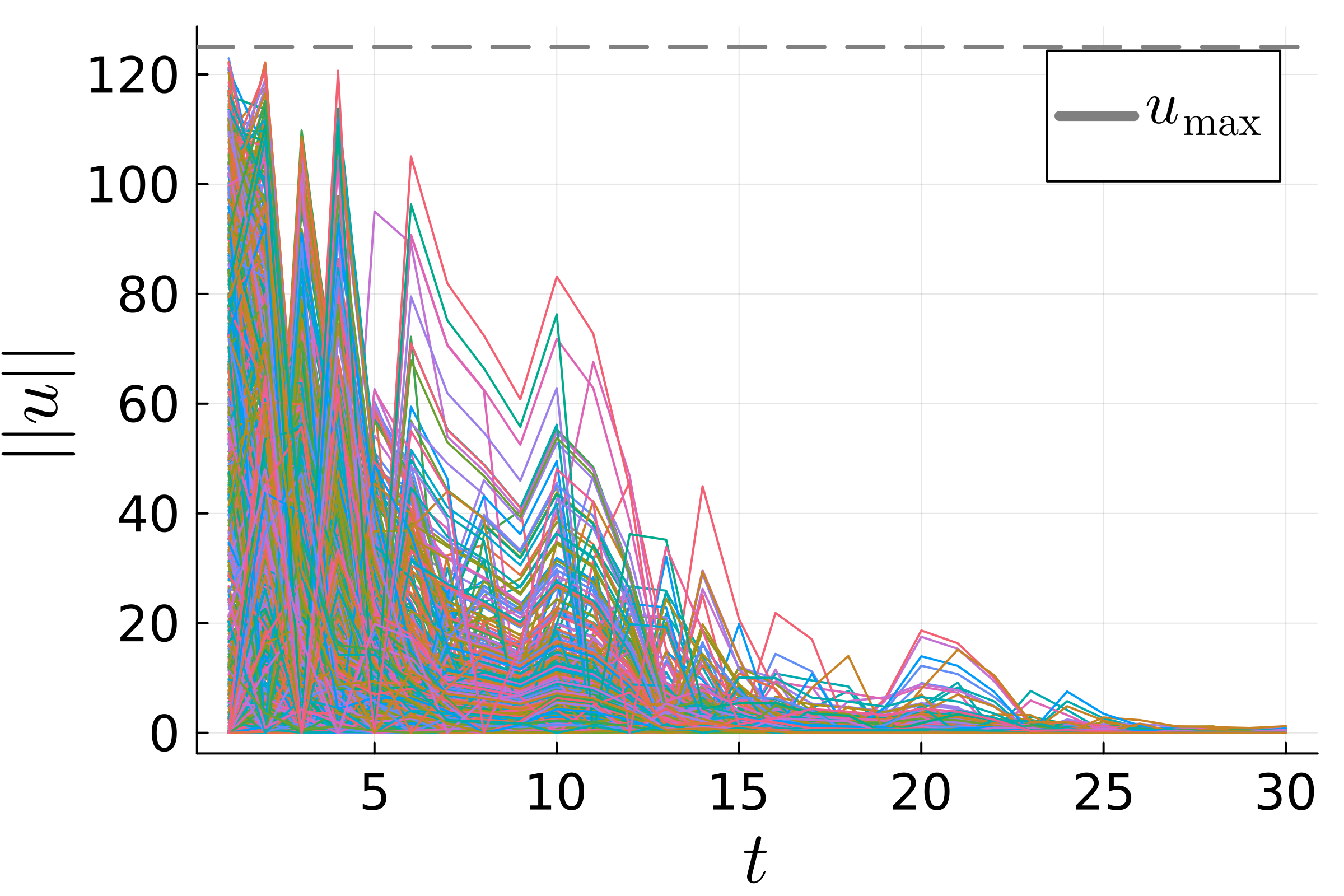}
         \caption{Control history}
         \label{fig:rel_uc_xhist}
     \end{subfigure}
     \hfill
     \caption{Problem 2 Monte Carlo simulations}
\end{figure}
\section{Conclusion}\label{conc_sec}
In this paper, we have considered robust controller synthesis problems for LTVPM. We propose a stability criterion for LTVPM in \Cref{thm4} and then relax it in \Cref{thm_rel} by leveraging periodicity of the system. In \Cref{sec_framework}, we utilize the stability criteria to formulate LMI-based controller synthesis frameworks that allow for state and control constraints with probability one satisfaction. Numerical results for both frameworks are presented.

\bibliographystyle{ieeetr}
\bibliography{references}

\end{document}